\documentclass[12pt]{amsart}
\usepackage{latexsym,amssymb}
\usepackage{graphicx}
\usepackage{amsthm}
\usepackage{mathrsfs}
\usepackage{amsmath}
\usepackage{xcolor}
\newcommand{\tends}[1]{\mbox{\space \raise-2mm
\hbox{$\textstyle\longrightarrow\atop\scriptstyle {#1}$} \space}}
\newcommand{\Id}{\operatorname{Id}}
\newcommand{\one}{\pmb{1}}
\newcommand{\st}{\,;\;} 
\newcommand{\cp}{\G\ltimes C(\Om)}
\newcommand{\G}{\Gamma}
\newcommand{\Om}{\Omega}

\newcommand{\inv}[1]{{#1}^{-1}}
\newcommand{\norm}[1]{\|{#1}\|}
\newcommand{\w}[1]{\widehat{#1}}

\newcommand{\mcJ}{\mathcal J}
\newcommand{\mcK}{\mathcal K}
\newcommand{\mcU}{\mathcal{U}}
\newcommand{\Aut}{\operatorname{Aut}}
\newcommand{\mcH}{\mathcal H}

\newcommand{\eps}{\varepsilon}
\newcommand{\mcD}{\mathcal{D}}

\newcommand{\N}{\mbox{${\mathbb{N}}$}}

\newcommand{\C}{\mbox{${\mathbb  C}$}}

\newcommand{\picap}{\widehat{\pi}}

\newcommand{\period}{\;.}
\newcommand{\Hcap}{\widehat{H}}
\newcommand{\tr}{\mathrm{tr}}
\newcommand{\ID}{\operatorname{Id}}
\newcommand\atopn[2]{\genfrac{}{}{0pt}{}{#1}{#2}}
\newcommand{\Ker}{\mathrm{Ker}}
\newcommand{\im}{\mathrm{Im}}

\newcommand{\dfn}{\emph}

\newtheorem{thm}{Theorem}[section]
 \newtheorem*{reftheorem1}{Theorem \reftotheorem}

 \newtheorem{lem}[thm]{Lemma}
 \newtheorem{prop}[thm]{Proposition}
 \theoremstyle{definition}
 \newtheorem{defn}[thm]{Definition}
 \newtheorem{dthm}[thm]{Duplicity Theorem}
  \newtheorem{othm}[thm]{Oddity Theorem}
 \theoremstyle{remark}
 \newtheorem{rem}[thm]{Remark}
 
\begin{document}

\title[vector-valued multiplicative functions, III ]
{Free group representations from vector-valued multiplicative functions, III }
\author{M. Gabriella Kuhn}
\address{ Dipartimento di Matematica e Applicazioni\\
 Universit\`{a} di Milano ``Bicocca'', Via Cozzi 53, 20125 Milano, ITALY }
\curraddr{}
\email{mariagabriella.kuhn@unimib.it}
\thanks{}

\author{Sandra Saliani}
\address{Dipartimento di Matematica, Informatica ed Economia\\
 Universit\`{a} degli Studi
 della Basilicata, Viale dell'ateneo lucano 10, 85100 Po\-ten\-za, ITALY}
\curraddr{}
\email{sandra.saliani@unibas.it}
\thanks{}

\author{Tim Steger}
\address{Matematica \\ Universit\`{a} degli Studi di Sassari,
Via Piandanna 4, 07100 Sassari, ITALY}
\curraddr{}
\email{steger@uniss.it}
\thanks{}


\date{\today}

\dedicatory{}

\begin{abstract}
Let $\pi$ be an irreducible unitary representation of a finitely
generated nonabelian free group $\G$; suppose $\pi$~is weakly
contained in the regular representation. In 2001 the first and third
authors conjectured that such a representation must be either
\dfn{odd} or \dfn{monotonous} or \dfn{duplicitous}. In 2004 they
introduced the class of \dfn{multiplicative representations}: this is
a large class of rep\-res\-enta\-tions obtained by looking at the action
of~$\G$ on its Cayley graph. In the second paper of this series we
showed that some of the multiplicative representations were
monotonous.  Here we show that all the other multiplicative
representations are either odd or duplicitous. The conjecture is
therefore established for multiplicative representations.
\end{abstract}

\keywords{Irreducible unitary representations, Free groups, Boundary realization, Duplicity, Oddity}

 \subjclass{43A65, 43A35, 15A42, 15B48, 22D25, 22D10}

\maketitle

\section{Introduction}
\label{intro}
Let $\G$~be a free group on a finite set of generators, $\Om$~its
Gromov boundary and~$C(\Om)$ the $C^*$-algebra of complex continuous
functions on~$\Om$. Given a unitary representation $(\pi,\mcH)$
of~$\G$, we say that $\iota$~is a boundary realization of $\pi$ if
$\iota$~is an isometric $\Gamma$-inclusion of~$\mcH$ into~$\mcH'$
where
  \begin{itemize}
  \item[$\bullet$] $\mcH'$ is the representation space of a
    $\cp$-representation~$\pi'$,
  \item[$\bullet$] $\iota(\mcH)$ is cyclic for the action of~$C(\Om)$
    on~$\mcH'$.
  \end{itemize}
One thinks of the map~$\iota$ as an identification of~$\mcH$ with a
subspace of an $L^2$-space on~$\Om$, where the $L^2$-space carries a
$\Gamma$-action compatible with the $\Gamma$-action on~$\Om$ and
the action of~$C(\Om)$ is by pointwise multiplication.

For brevity we shall say that a representation is \dfn{tempered} if it
is weakly contained in the regular representation. Note that $\pi$~has
a boundary realization if and only if it is tempered (see
\cite{HKS19}).  Observe that $\iota$~is an isometry, but it needn't be
surjective!  Call a boundary realization $\iota:\mcH\to\mcH'$
\dfn{perfect} if~$\iota$ is a unitary equivalence, i.e. a bijection
and not just an injection.

Fix a tempered irreducible unitary representation~$\pi$ of~$\G$.  We
say that
$\pi$~satisfies \emph{monotony\/} if
\begin{itemize}
\item[$\bullet$] Up to equivalence $\pi$~has a single boundary realization.
\item[$\bullet$] That realization is perfect.
\end{itemize}
$\pi$~satisfies \emph{duplicity\/} if
\begin{itemize}
\item[$\bullet$] Up to equivalence $\pi$~has exactly two perfect boundary
realizations, $(\iota_1,\pi_1')$ and $(\iota_2,\pi_2')$.  \item[$\bullet$]
Let~$\pi'$ acting on $\mcH'_1\oplus\mcH'_2$ be the direct sum of~$\pi_1'$
and~$\pi_2'$.  Up to equivalence all the imperfect boundary
realizations of~$\pi$ are given by the maps $v\mapsto
(t_1\iota_1(v),t_2\iota_2(v))$ where $t_1,t_2>0$ and $t_1^2+t_2^2=1$.
\end{itemize}
$\pi$~satisfies \emph{oddity\/} if
\begin{itemize}
\item[$\bullet$] Up to equivalence $\pi$~has exactly one boundary realization,
$(\iota,\pi')$ where $\pi'$~acts on~$\mcH'$.
\item[$\bullet$] That realization is not perfect.
\end{itemize}

The principal series representations considered by Fig\`a-Talamanca
and Picardello and Fig\`a-Talamanca and Steger \cite{FTP1}, \cite{FTS}
satisfy duplicity, except for the two at the endpoints, which satisfy
monotony.  Examples of irreducible representations satisfying oddity
can be found in the papers of Paschke (see~\cite{Pa1}, \cite{Pa2}) and
Example~6.5 of~\cite{KS04}.

In \cite{KS04} the first and third authors constructed a class of
\emph{multiplicative representations}.  This class is large enough to
contain more or less all previously constructed tempered irreducible
representations whose construction uses the action of~$\G$ on its
Cayley graph. In the same paper they proved that those representations
are irreducible \emph{as representations of} $\cp$.  The second paper
in this series~\cite{KSS16} and the present third paper are devoted to
the study of irreducibility and inequivalence of multiplicative
representations as \emph{representations of} $\G$.

It turns out that the growth of matrix coefficients plays a central
role both for the classification above conjectured and for the proof of
irreducibility.

In~1979 Haagerup~\cite{H} showed that tempered representations can be
characterized by the growth of their matrix coefficients, namely he
proved that for a representation~$\pi$ having a cyclic vector~$v$ the
following conditions are equivalent:
\begin{itemize}
\item[$\bullet$] $\pi$ is tempered;
\item[$\bullet$] The map $\phi_\eps^v(x)=\langle v,\pi(x)v\rangle e^{-\eps|x|}$
is square integrable for every positive $\eps$;
\item[$\bullet$]
  \begin{equation}\label{haagerup}
\displaystyle{
  \sum_{|x|=n} |\langle v, \pi(x)v\rangle|^2\leq  (n+1)^2\|v\|^4.}
\end{equation}
\end{itemize}
The  third condition implies
\begin{equation}\label{haag}
\|\phi_\eps^v\|_2^2=
\sum_{x\in\Gamma} |\langle v,\pi(x)v\rangle|^2e^{-2\eps|x|}
\leq C\|v\|^4 \left(\frac1{\eps}\right)^3
\period
\end{equation}
The problem of finding the correct asymptotic for $\sum_{|x|=n}
|\langle v, \pi(x)v\rangle|^2$ is nontrivial and has been treated by
many authors, not only for a free group: see for example \cite{BG19},
\cite{BLP}, \cite{BP}, \cite{FTP1}, \cite{FTS}, \cite{KSS16}.

In \cite{KSS16} it is shown if $\pi$~is a multiplicative
representation then $\|\phi_\eps^v\|_2^2$ can be explicitly calculated
for $v$~in a dense set of \dfn{smooth vectors}. It turns out to be
asymptotically proportional to either $1/\eps$ or $1/\eps^2$ or
$1/\eps^3$ as~$\eps\to 0$.

The exponent~$3$ for~$1/\eps$ in Haagerup's inequality \eqref{haag} is
an upper bound for the growth of the $\ell^2$~norm of~$\phi_\eps^v$
which is attained only in rather special cases: endpoint
representations of the isotropic/an\-iso\-tropic principal series of
Fig\`a-Talamanca and Picardello/Fig\`a-Ta\-la\-man\-ca and Steger
\cite{FTP1}, \cite{FTS} have these maximal asymptotics; likewise for
other multiplicative representations constructed from very special
matrix systems (see \cite{KSS16}).  Recently Boyer and Garncarek
\cite{BG19} constructed a huge family of irreducible representations
with maximal asymptotics.

Every multiplicative representation provides a perfect boundary
realization of itself (see Proposition \ref{perfect-itself}), but what
happens when we restrict this representation to~$\G$? Is this
representation still irreducible? Are there other boundary realization
of this $\G$-representation?

In \cite{KS01} we conjectured that any irreducible unitary
representation of~$\G$ weakly contained in the regular representation
is monotonous, or duplicitous, or odd.

In \cite{KSS16} we characterized, within the class of multiplicative
representations, those which satisfy monotony: they are exactly those
for which either $\|\phi_\eps^v\|_2^2\simeq 1/\eps^2$ or
$\|\phi_\eps^v\|_2^2\simeq 1/\eps^3,$ as $\eps\rightarrow 0.$ This
paper is devoted to the study of the case
$$\|\phi_\eps^v\|_2^2\simeq 1/\eps\,.$$ We shall prove that in this
case there are only two possibilities:
\begin{itemize}
\item[$\bullet$]  Either the multiplicative $\G$-representation is
  irreducible and satisfies duplicity (Theorem \ref{irre});
\item[$\bullet$] or the multiplicative $\G$-representation decomposes into
  two \emph{twin} irreducible $\G$-representations and each of them
  satisfies oddity (Theorem \ref{irre-odd}).
\end{itemize}
We may conclude that our conjecture is true for the irreducible
components of multiplicative $\G$-representations.

The techniques used here are completely different from those of
\cite{KSS16}: we shall use the Duplicity and Oddity Theorems of
\cite{HKS19} and new investigations into the eigenspace of~$1$ of the
transition matrix used to compute $\|\phi_\eps^v\|_2^2$.

\section{Preliminary}

$\G$ will stand for a non-abelian free group on a finite  set $A^+$  of generators.
We also let $A=A^+\cup A^-$ for the set of generators and their inverses.
Recall that the Cayley graph of~$\Gamma$ with respect
to~$A$ is a tree, and that this tree has a standard compactification
which is obtained by adjoining a boundary, which we
denote~$\Om$. This boundary can be described as the space of ends of
the tree; it also coincides with the boundary of~$\Gamma$ considered
as a Gromov hyperbolic group.

 Concretely, if we identify~$\Gamma$ with
the set of finite reduced words
$$
  \{ a_1a_2\dots a_n \st a_j\in A, a_ja_{j+1}\neq 1\},$$
then we can identify~$\Om$ with the set of infinite reduced words
$$
  \{ a_1a_2a_3\dots  \st a_j\in A, a_ja_{j+1}\neq 1\}.$$

For $x\in\Gamma$ let~$\Gamma(x)$ be the set of finite reduced words
which start with the reduced word for~$x$; let~$\Om(x)$ be the set of
infinite reduced words which start with the reduced word for~$x$. A
basis for the topology on the compactification $\Gamma\sqcup\Om$ is
given by the singletons $\{x\}$ and the sets $\Gamma(x)\sqcup\Om(x)$,
as $x$~varies through~$\Gamma$.  The left-action of~$\Gamma$ on
$\Gamma$ extends to a continuous action on the compactification.

For any directed edge $(x,xa)$ of the tree define
$$
\Gamma(x,xa)=\{ y\in\Gamma \st d(y,xa)<d(y,x) \},
$$
where $d$ counts the number of the edges
joining two vertices.
Upon removing that edge, the tree decomposes as
$\Gamma=\Gamma(x,xa)\amalg\Gamma(xa,x)$.

 Suppose that the length $|xa|=|x|+1$,
that is, suppose that~$a$ is the last letter in the reduced word
for~$xa$.  Then $\Gamma(x,xa)=\Gamma(xa)$ (while if $|xa|=|x|-1$, then
$\Gamma(x,xa)=\Gamma\sim\Gamma(x)$).  If $|y|<|xa|=|x|+1$, then also
$|yxa|=|yx|+1$, and so
$$
\begin{aligned}
y\Gamma(xa)&= y\Gamma(x,xa)=\Gamma(yx,yxa)=\Gamma(yxa) \\
y\Omega(xa)&=y(\overline{\Gamma(xa)}\cap\Omega)=\overline{y\Gamma(xa)}\cap\Omega
  =\overline{\Gamma(yxa)}\cap\Omega
  =\Omega(yxa)
\period
\end{aligned}
$$

Considering, by contrast, the case $y=(xa)^{-1}$, one finds
$$(xa)^{-1}\Gamma(xa)=(xa)^{-1}\Gamma(x,xa)
=\Gamma(a^{-1},e)=\Gamma\sim\Gamma(a^{-1}).$$

In order to construct a multiplicative representation as
described in \cite{KS04} one needs:

$\bullet\;\;$
  A matrix system $(V_a, H_{ba})$: it consists of
finite dimensional complex vector spaces $V_a,$ for each $a\in A,$ and linear maps
$H_{ba}:V_a\rightarrow V_b$ for each pair $a,b\in A,$ where $H_{ba}=0$
whenever $ba=e.$

$\bullet\;\;$
 A collection of positive definite sesquilinear forms $(B_a)$ on each $V_a$
  satisfying, for each $a\in A$ and  $v_a\in V_a$ the following \emph{compatibility condition:}
  \begin{equation}\label{compatibility}
B_a(v_a,v_a)=\sum_{b\in A}{B_b( H_{ba}v_a, H_{ba}v_a) }.
  \end{equation}

$\bullet\;\;$ A space $\mcH^\infty$ of \emph{multiplicative functions}
  obtained as follows:
  for $x\in\Gamma$ and
$a\in A$ let
  \begin{equation}\label{mu}
    \left\{
\begin{array}{ll}
 & \mu[x,x a, v_a](y)=0,\quad \mbox{for $y\notin\Gamma(x,xa)$},\\
&\mu[x,x a, v_a](xa)=v_a,\\
  &\mu[x,x a, v_a](ybc)=H_{cb}\, \mu[x,x a, v_a](yb),\\
  &\mbox{if $yb,ybc\in\Gamma(x,xa),$ and
$d(ybc,x)=d(y,x)+2.$}
\end{array}\right.
  \end{equation}

  Thus if $ab_1\dots b_n$ is reduced
\begin{equation*}
\mu[x,{xa},{v_a}](xab_1\dots b_n)=
 H_{b_n b_{n-1}}\dots H_{b_2b_1}H_{b_1a}v_a.
\end{equation*}

In this definition we do not require that $|xa|=|x|+1$.
For example, if $x=a_1\dots a_n$ is the reduced word for $x$ and
$a=\inv{a_n}$, then $e\in\G(x,xa)$ and
$$
\mu[x,x\inv{a_n},v_{\inv{a_n}}](e)= H_{\inv{a_1}\inv{a_2}}\dots
H_{\inv{a_{n-1}}\inv{a_n}}v_{\inv{a_n}}\,.
$$

Note that $y\G(x,xa)=\G(yx,yxa)$ and that
$$
\mu[x,xa,v_a](\inv y\cdot)=\mu[yx,yxa,v_a](\cdot)\,.
$$

A \dfn{multiplicative function} is a finite linear combination of such functions $\mu[x,xa,v_a]$
as $x$, $a$ and $v_a$ vary in $\G$, $A$ and $V_a$. Say that two multiplicative
functions are equivalent if they differ only on a finite set.
$\mcH^\infty$ is the quotient space of the space of multiplicative functions
with respect to this equivalence relation.

$\bullet\;\;$ An inner product $\langle \cdot,\cdot \rangle  $ on $\mcH^\infty.$
Define
\begin{equation}\label{inner}
  \begin{aligned}
    &\langle\mu[x,x a, v_a],\mu[y,yb, v_b]\rangle
    =0,\quad\mbox{if $\G(x,xa)\cap\G(y,yb)=\emptyset$},\\
    &\langle\mu[x,xa,v_a],\mu[x,xa,v'_a]\rangle
   =B_a(v_a,v'_a).
    \end{aligned}
\end{equation}
The compatibility condition \eqref{compatibility} ensure that $\langle\cdot,\cdot\rangle  $
is well defined.

Since every $f\in\mcH^\infty$  for $N$ large enough, and modulo the
equivalence relation, can be written as
\begin{equation}\label{effe}
f=\sum_{|x|=N}\sum_{\atopn{a\in A}{|xa|=N+1}}\mu[x,xa,f(xa)]
\end{equation}
$\langle\cdot\,,\cdot\rangle  $
extends by linearity to a scalar product in $\mcH^\infty$ (see \cite{KS04} for details).

Given all these ingredients one sets, for every $y\in\G$ and $f\in\mcH^\infty,$

  \begin{equation*}
    (\pi(y)f)(x)=f(\inv y x).
  \end{equation*}

 It can be shown that $\pi$ is unitary with respect to $\langle~\cdot,\cdot~\rangle$
and hence extends to $\mcH$, the completion of $\mcH^\infty$,
  to a unitary representation that we shall call \emph{multiplicative}.
We aware the reader that
the multiplicative
representation hitherto constructed need not be \emph{irreducible}, as it
is pointed out in \cite{IKS}. In order to hope for an irreducible representation of $\G$
we need to impose the following irreducibility condition on the matrix system:
\begin{defn}
An invariant subsystem of $(V_a, H_{ba})$ is a collection of subspaces
$W_a\subseteq V_a$ such that $H_{ba}(W_a)\subseteq W_b,$ for all $a,b\in A.$
The system $(V_a, H_{ba})$ is called \dfn{irreducible} if it is non-zero and there are no
invariant subsystems except for itself and the zero subsystem.
\end{defn}
In \cite{KS04} it is proved that any irreducible matrix system admits,
up to a normalization, a unique (up to scalar multiples) tuple of
strictly positive definite forms $(B_a)$ satisfying
\eqref{compatibility}.


From this point on we shall assume that {all systems are irreducible and normalized} so that
\eqref{compatibility} holds for a given
tuple $(B_a)$ of positive definite forms.

For brevity we shall call such  systems
 $(V_a, H_{ba}, B_a)$
 \emph{matrix systems with inner product}.

Now we need to specify in which sense a multiplicative representation
gives rise to a $\G\ltimes C(\Om)$ representation.
\begin{prop}\label{perfect-itself}
  Let $(\pi,\mcH)$ be a multiplicative representation   as described above.
  Then $(\operatorname{Id},\mcH)$ is a perfect boundary realization
  of $\pi$.
\end{prop}
\begin{proof}
Let $f\in\mcH$, $x\in\G$ and let $\one_x$ be the characteristic function of the set
$\Om(x)$. Set
\begin{equation*}
  (\pi(\one_x)f)(y)=\left\{\begin{aligned}
      &f(y),\quad\mbox{if $y\in\G(x)$},\\
      &0,\quad\mbox{otherwise},
    \end{aligned}\right.
\end{equation*}
and extend this action by linearity and continuity to all of $C(\Om)$.

Observe that one has
\begin{equation}\label{bdry-rpn}
  \pi(x)\pi(G)\pi(x)^{-1} = \pi(\lambda(x)G),
    \qquad\text{for $x\in\Gamma$ and $G\in C(\Om)$.}
\end{equation}
where
$\lambda:\Gamma\to\Aut(C(\Om))$ is given by
$ (\lambda(x)G)(\omega)=G(x^{-1}\omega)$.
In fact, a pair of actions which satisfy~\eqref{bdry-rpn} fit together
to give a representation of the
\emph{crossed-product $C^*$-algebra}, denoted~$\cp$.

Vice versa, any
$\cp$-representation comes from a pair of actions which fit together
as per~\eqref{bdry-rpn}.  One can consult \cite{Davidson} for the
definition of the crossed-product.
Hence  $(\operatorname{Id},\mcH)$ is a boundary realization of $\pi$
which is obviously surjective.
\end{proof}


Hence any multiplicative representation provides a perfect boundary
realization of itself.
Suppose $\pi$~is \emph{irreducible} and weakly contained in the
regular representation. How many different boundary realizations does
it have?  To make this question precise, we need
\begin{defn}\label{equiv-cp}
  Let $\pi:\Gamma\to\mcU(\mcH)$ be a unitary representation
  of~$\Gamma$. Two boundary realizations $\iota_j:\mcH\to\mcH_j'$ are
  \emph{equivalent} if there is a unitary equivalence of
  $\cp$-representations $J:\mcH_1'\to\mcH_2'$, such that
  $J\iota_1=\iota_2$.
\end{defn}
The above definition concerns equivalence of $\cp$ representations, while the
next definition concerns  the \emph{equivalence of matrix systems:}
\begin{defn}\label{equiv-sys}
A \emph{map } from the system
   $(V_a, H_{ba})$ to
$(V_a^{\sharp}, H_{ba}^{\sharp}),$  is a tuple $(J_a),$ where
$J_a:V_a\rightarrow V_a^{\sharp},$ is a linear map and
$$
H_{ba}^{\sharp} J_a=J_b H_{ba} \quad\mbox{for all $a, b\in A$}\,.
$$
The tuple is called an \emph{equivalence} if each $J_a$ is a bijection.
Two systems are called \emph{equivalent } if there is an equivalence between
them.
\end{defn}

In \cite{KS04} it is proved that
two multiplicative representations arising from irreducible matrix systems
are equivalent as $\cp$ representations  if and only if the two systems
are equivalent; hence for irreducible $\cp$ representations the two definitions
\ref{equiv-cp} and \ref{equiv-sys} are
equivalent.
Obviously equivalent systems give raise to equivalent $\G$-representations,
but the converse is no longer true as we shall see in Section \ref{equivG}.

\section{Good Vectors and Boundary Realizations.}

We begin with the following important
\begin{defn}
Given a representation
$(\pi,\mcH)$, we say that a non-zero vector $w\in \mathcal{H}$
is a \emph{good vector} if it
satisfies the \emph{Good Vector Bound}, namely  if
there exists a constant $C$, depending only on $w$, such that
\begin{equation}\label{GVBo}
\sum_{|x|=n}{|\langle v,\pi(x)w\rangle  |^2}\leq C\| v\|^2,\quad
\text{for all }\; n\in\N, v\in \mathcal{H}.\tag{GVB}
\end{equation}
\end{defn}
We observe that, if $w$ is a good vector, then, for every $v,$
as $\eps\rightarrow 0,$
$$\sum_{x\in\G}{|\langle v,\pi(x)w\rangle|^2}e^{-\eps|x|}=
\sum_{n=0}^\infty \displaystyle\sum_{|x|=n}{|\langle v,\pi(x)w\rangle  |^2}e^{-\eps n}
\leq C\frac{\| v\|^2}{1-e^{-\eps}}\simeq\frac1\eps.
$$

The first key observation is that
the existence of ``good vectors'' is deeply related to the existence
of imperfect boundary realizations, as  the following proposition shows:
\begin{prop}\label{impimp}\cite{KS01}
If a representation $(\pi,\mathcal{H})$ of $\G$ admits an \emph{imperfect}
boundary realization, then some non-zero vector $w\in \mathcal{H}$ satisfies
\eqref{GVBo}.
\end{prop}

For an arbitrary representation it is really very hard to understand whether
there exists or not a good vector! See for example the paper of
Boyer and Garncareck \cite{BG19} where the existence of good vectors is ruled out!

For a representation constructed from an irreducible matrix system there is the possibility to calculate
\begin{equation}\label{somme-generic}
  \sum_{x\in\G}|\langle v,\pi(x)w\rangle|^2  e^{-\eps|x|}
\end{equation}
explicitly
at least for a dense set of vectors $v$ and $w$.
The calculations in the following subsection are taken from \cite{KSS16}.

\subsection{The Twin of a System}

Given a finite dimensional vector space $V$, $\overline{V}$ will stand for
its complex conjugate, $V'$ for its dual, the space of linear functionals,
while
 $V^{\ast}=\overline{V}'$ will stand for the space of antilinear functionals.
We recall some identifications that will be used  in this paper.
$V_1\otimes V_2$ will be identified with the space of linear maps
$v_1\otimes v_2:V_2'\rightarrow V_1$ given by
$
  (v_1\otimes v_2)(f)=f(v_2)\, v_1=\,\langle v_2,f\rangle\, v_1\;.
$

It follows that, given $T_1\in\mathscr{L}(V_1,V_3)$ and
$T_2\in\mathscr{L}(V_2,V_4)$, the map $$T_1\otimes T_2:V_1\otimes V_2\rightarrow V_3\otimes V_4$$
corresponds to the operator
$$
 \mathscr{L}(V_2',V_1)\rightarrow \mathscr{L}(V_4',V_3),
\quad S\mapsto T_1\,S\,T_2'.
$$
So we shall write $(T_1\otimes T_2)S=T_1S T_2'.$ The duality isomorphism
$${\mathcal L}:\mathscr{L}(V,V^{\ast})\rightarrow\mathscr{L}(V^{\ast},V)'$$
defines a bilinear form which can be written explicitly by means of the trace function
$$\mathcal{B}: \mathscr{L}(V,V^{\ast})\times\mathscr{L}(V^{\ast},V)\rightarrow \C,$$
\begin{equation}\label{otto}
\mathcal{B}(T,S):=({\mathcal L}(T))(S)=\mathrm{tr}(TS)=\mathrm{tr}(ST).
\end{equation}

Positive definite sesquilinear forms $B_a$
on $V_a$ are identified with
maps $B_a \in\mathscr{L}(V_a,V_a^{\ast})$;  under this identification one also has $B^*_a=B_a$.

The compatibility condition \eqref{compatibility} may be rewritten as:
\begin{equation}\label{ti}
(TB)_a=
\sum_{b\in A}H^\ast_{ba}B_bH_{ba}=\sum_{b\in A} H^\ast_{ba}\otimes H_{ba}'B_b=B_a.
\end{equation}

The above equation says that the tuple $(B_a)$
is a right eigenvector for
the matrix $T=(H^\ast_{ba}\otimes H_{ba}')_{a,b}$
corresponding to eigenvalue $1$.

For every $a\in A$ set
$
\widehat{V}_a:=V^{\ast}_{a^{-1}}=\overline{V}_{a^{-1}}'.
$

A system of linear maps  $H_{ba}:V_a\to V_b$ induces an obvious system of
maps $H^{\ast}_{ba}:V_b^{\ast}\rightarrow V_a^{\ast},$
by the rule
$H^{\ast}_{ba}(f)=f\circ H_{ba},$ and also maps
$$
\widehat{H}_{ba}:=H^{\ast}_{a^{-1}b^{-1}}:\widehat{V}_a\rightarrow\widehat{V}_b.
$$

Hence the matrix system $(V_a, H_{ba})$ induces
another matrix system $(\widehat{V}_a, \widehat{H}_{ba}),$ which is irreducible if
so is $(V_a, H_{ba}).$

\begin{prop}\cite{KSS16}
Assume that  $(V_a, H_{ba}, B_a)$ is a matrix  system
with inner product. Then there exists a unique (up to multiple
scalars) positive definite tuple $(\widehat{B}_a),$
$\widehat{B}_a:\widehat{V}_a\rightarrow\widehat{V}_a^{\ast}$
on $\widehat{V}_a$
such that the matrix system
$(\widehat{V}_a,\widehat{H}_{ba},\widehat{B}_a)$ is
a system with inner products.
\end{prop}

\begin{defn}
  The system $(\widehat{V}_a,\widehat{H}_{ba},\widehat{B}_a)$
  above constructed is called the \dfn{twin of the system} $(V_a,H_{ba},B_a).$
\end{defn}

The twin of the system will play a central role in the computation
of \eqref{somme-generic}.
Let $\mu[e,a,v_a]$ and $\mu[e,b,v_b]$ be elementary multiplicative functions
constructed from the system $(V_a,H_{ba},B_a)$ as per \eqref{mu} and let
$(\widehat{V}_a,\widehat{H}_{ba},\widehat{B}_a)$
be the twin system.
For any $a,b\in A,$ we define  maps $E_{ab}:V_b\rightarrow \widehat{V}_a$ by
$$
E_{ab}=\sum_{\atopn{c\in A}{ c\neq a, b^{-1}}}{{H}^{\ast}_{c a^{-1}}{B}_c{H}_{cb}}=
\sum_{\atopn{c\in A}{ c\neq a, b^{-1}}}{\widehat{H}_{a c^{-1}}{B}_c{H}_{cb}},
$$
where $E_{ab}=0$ whenever $ab=e$.

The quantity
$$
 \|\phi_\eps^{v_a,v_b}\|_2^2 = \sum_{x\in\Gamma}
{|\langle \mu[e,a,v_a],\pi(x)\mu[e,b,v_b]\rangle  |^2}
e^{-\eps|x|}
$$
can be calculated using the following block matrix:
$
{\mathcal D}=(D_{i,j})_{i,j=1,\dots, 4}$
\begin{equation}\label{matrix}
{\mathcal D}=\!\!\left(\begin{array}{cccc}
\left(\!{\widehat{H}}_{ab}\otimes \overline{{\widehat{H}}}_{ab}\!\right)_{a,b}
&\!\!\left(\!E_{ab}\otimes \overline{{\widehat{H}}}_{ab}\!\right)_{a,b}
&\!\!\!\left(\!{\widehat{H}}_{ab}\otimes \overline{{E}}_{ab}\!\right)_{a,b}
&\!\!\!\left(\!{E}_{ab}\otimes \overline{{E}}_{ab}\!\right)_{a,b}\\ \\
\!\!0
&\!\!\!\left(\!{H}_{ab}\otimes \overline{{\widehat{H}}}_{ab}\!\right)_{a,b}
&\!\!\!\!0
&\!\!\!\left(\!{H}_{ab}\otimes \overline{{E}}_{ab}\!\right)_{a,b}\\ \\
\!\!0
&\!\!\!\!0
&\!\!\!\left(\!{\widehat{H}}_{ab}\otimes \overline{{H}}_{ab}\!\right)_{a,b}
&\!\!\!\left(\!{E}_{ab}\otimes \overline{{H}}_{ab}\!\right)_{a,b}\\ \\
\!\!0
&\!\!\!0
&\!\!\!\!0
&\!\!\!\left(\!{H}_{ab}\otimes \overline{{H}}_{ab}\!\right)_{a,b}
\end{array}
\right)
\end{equation}
Indeed $\mcD$ is a transition matrix
which allows to pass from
$$\sum_{|x|=n}
{|\langle \mu[e,a,v_a],\pi(x)\mu[e,b,v_b]\rangle|^2}
$$
to the same quantity summed on all $|x|=n+1.$

It can be shown that, under our assumptions, the spectral radius of $\mcD$ is
always one, \cite{KS96}. Moreover
\begin{itemize}
\item[A)] If the two systems
  $(H_{ba}, V_a,B_a)$ and
  $(\widehat{V}_a,\widehat{H}_{ba},\widehat{B}_a)$
  are  {not equi\-val\-ent},
  $1$ is an eigenvalue of multiplicity {two};
\item[B)] If the two systems
$(H_{ba}, V_a,B_a)$ and
  $(\widehat{V}_a,\widehat{H}_{ba},\widehat{B}_a)$
  are {equi\-val\-ent}, $1$ is an
 eigenvalue of multiplicity {four}.
\end{itemize}

Hence the growth of the quantity
$\|\phi_\eps^{v_a,v_b}\|_2^2$
depends on the eigenspace of $1$ of $\mcD$.
This space can have, in general, dimension $1,$ $2,$ $3$ or $4,$
depending on many facts.

The study for inequivalent systems has been done in \cite{KSS16} and it is summarized in the
following Theorem:
\begin{thm}\label{pregressi}
If the two systems  $(V_a,H_{ba},B_a)$ and
 $(\widehat{V}_a,\widehat{H}_{ba},\widehat{B}_a)$
{are not equivalent}, then
\begin{itemize}
  \item[\textbf{AI}] The dimension of the eigenspace of $1$
is $2$ if and only if, for every $a,b,\in A$ and $v_a,v_b\in V_a,V_b$, one has
$$\displaystyle{\|\phi_\eps^{v_a,v_b}\|_2^2\simeq\frac1\eps},\quad
\text{as}\quad \eps\rightarrow 0.$$
%
In this case there exists a unique  tuple of linear maps
$Q_a:V_a\to\widehat{V}_a$ satisfying
\begin{equation}\label{imain}
  {\widehat{H}}_{ab}Q_b+ E_{ab}=Q_a H_{ab},\quad  {a,b\in A}.
\end{equation}
\item[\textbf{AII}]
The dimension of the eigenspace of $1$ is $1$ if and only if
 for every $a,b,\in A$ and $v_a,v_b\in V_a,V_b$, one has
 $$\displaystyle{\|\phi_\eps^{v_a,v_b}\|_2^2\simeq\frac1{\eps^2},}
 \quad
\text{as}\quad \eps\rightarrow 0.$$
%
In this case  {no vector in $\mcH$}
satisfies the good vector bound \eqref{GVBo}
and no system of $Q_a:V_a\to\widehat{V}_a$ can satisfy \eqref{imain}.
\end{itemize}
\end{thm}
\subsection{Equivalent Systems}
Now we pass to the study of equivalent systems, therefore we shall assume, till the end
of this section, that the two systems $(V_a, H_{ba},B_a)$
and $(\w{V}_a,\w{H}_{ba},\w{B}_a)$ are irreducible, normalized, and {equivalent}.

 We are interested in the case where the dimension of the eigenspace of $1$ is $4,$
since this is linked to the growth $1/\eps,$ as shown in the following

\begin{thm}\cite[Theorem 1]{KSS16}
Let $\mcD$ be the matrix constructed by equivalent systems as in \eqref{matrix}
  and let $d$ be the dimension of
the eigenspace of $1$ .
 Then $d=4$ if and only if
$$\|\phi_\eps^{v_a,v_b}\|_2^2\simeq\frac1\eps,\quad
\text{as}\quad \eps\rightarrow 0.$$
\end{thm}

and the result follows.

\begin{rem}\label{Kauto}
Let $(K_a)_a,$  be an equivalence between the two systems
\begin{equation}\label{equivK}
K_a:V_a\rightarrow \w{V}_a,\quad \w{H}_{ab}K_b= K_a H_{a b},\quad \text{for all}\, a,b\in A.
\end{equation}

One can always assume $K_a^{\ast}=K_{\inv a}.$
Indeed, passing to adjoint in \eqref{equivK}, we get
\begin{equation*}
  H^*_{ab}K_a^*=K^*_b\w{H}_{ab}^*,\quad\mbox{that is}\quad
  \w{H}_{\inv b\inv a}K^*_a=K^*_b H_{\inv b\inv a}\,.
  \end{equation*}

Write $a$ and $b$ for $\inv a$ and $\inv b$, to see that
the tuple $K^*_{\inv a}$ gives another valid system. Since
$$
K_a=\frac{(K_a+K_{\inv a}^*)}2 + i\frac{(K_a-K_{\inv a}^*)}{2i},
$$
anyone of the two addends will satisfy the required condition.
\end{rem}


%
\begin{lem}\label{autodim4}
The following tuples
are, up to constants, the only right eigenvectors corresponding to the eigenvalue $1$
 for the indicated submatrices of the main matrix \eqref{matrix}
\begin{enumerate}
\item $U_1=(B_{\inv a})_a$  for $D_{11};$
\item $U_2=(K_a^{-1} B_{\inv a})_a$ for $D_{22};$
\item $U_3=(B_{\inv a} K_{\inv a}^{-1})_a$ for $D_{33};$
\item $U_4=(\w{B}_{\inv a})_a=(K_a^{-1} B_{\inv a}K_{\inv a}^{-1})_a$ for $D_{44}.$
\end{enumerate}
\end{lem}
\begin{proof}
 The result follows by  the compatibility condition \eqref{compatibility},
the property \eqref{equivK} of the equivalence $(K_a),$ and the identity $K_a^{\ast}=K_{\inv a}.$
\end{proof}

Now, \eqref{ti} shows that ${V}=(B_a)_a$ is an eigenvector corresponding to $1$ for the dual of $D_{44}$
$$D_{44}'=\left( (H_{ab}\otimes\overline{H_{ab}})_{a,b}\right)'=
(H_{ba}^{\ast}\otimes H'_{ba})_{a,b}.$$
Moreover, the space $\Ker((D_{44}-I)')$ is one-dimensional and generated by ${V}.$
The  equality
$\im(D_{44}-I)=\Ker((D_{44}-I)')^{\bot}$ therefore yields the following property
involving the trace \eqref{otto}, for any
arbitrary $X$:
\begin{equation}\label{imk}
X\in\im(D_{44}-I)\Leftrightarrow \tr(X {V})=\tr({V}X)=0.
\end{equation}


\begin{lem}\label{dim>3} Under the hypotheses hitherto assumed on the matrix systems,
the following are equivalent:
\begin{itemize}
\item[a)] The dimension of the eigenspace corresponding to the eigen\-va\-lue $1$ of
$\mcD$ is $d\geq 3;$
\item[b)]
 $\displaystyle{\sum_{a,b\in A}{\tr(K_a^{-1} E_{a b}\, \w{B}_{\inv b} H_{a b}^{\ast}\, B_a)}}=0.$
 \end{itemize}
\end{lem}
\begin{proof}
In case of equivalence of matrix systems, the equality \eqref{equivK} yields a similarity
between all elements in the main diagonal in ${\mathcal{D}},$  indeed
$$
D_{ii}=S_i\, D_{44}\,S_i^{-1}\quad i=1,2,3,$$
where  $S_1=\text{diag}(K_a\otimes \overline{K_a},\,a\in A),$
$\quad S_2=\text{diag}(\Id\otimes \overline{K_a},\,a\in A),$ and $S_3=\text{diag}(K_a\otimes \Id,\,a\in A).$

Hence the full matrix will be similar to a matrix having all the diagonal
elements equal to $D=D_{4,4}$.

 To analyze  the dimension of the eigenspace corresponding to the eigen\-va\-lue $1$ of
$\mcD$  means to find all solutions of $\mcD P=P,$  therefore we can assume that $\mathcal{D}-I$ is equal to

$$\left(\begin{array}{cccc}
D-I &A&B&C\\
0&D-I&0&B\\
0&0&D-I&A\\
0&0&0&D-I
\end{array}\right),$$
where  $A= S_3^{-1} D_{34},$ and $B=S_2^{-1} D_{2 4}.$

Let $Z=(Z_a)_a$ denote a (normalized) eigenvector corresponding to the eigen\-va\-lue $1$  of $D,$
(note that $Z_a^{\ast}=Z_a$);
   it  verifies  the equality
$$
(A\,Z)_a^{\ast}=\sum_{b}{(K_a^{-1} E_{a b}Z_b H_{a b}^{\ast})^{\ast}}
= \sum_{b}{H_{a b}Z_b (K_a^{-1} E_{a b})^{\ast}}=(B\,Z)_a.
$$

The eigenspace corresponding to the eigen\-va\-lue $1$  of ${\mathcal D}$  has dimension greater
or equal to $3$
 if and only if there exist $2$ linear independent eigenvectors $0\neq P=(P_j)_{j=1,\dots,4},$
not proportional to $u=(\delta_{j1}Z)_{j=1,\dots, 4},$  such that
${\mathcal D} P=P$.

If $P_2=P_3=P_4=0,$ and ${\mathcal D} P=P$ then $P_1$ is either zero or proportional to $Z.$
 Hence $P$  not proportional to $u$ means $(P_2,P_3,P_4)\neq (0,0,0).$

Observe that, for the discussion  before \eqref{imk},
 b) means $\tr(AZ{V})=0$ or, equivalently, $AZ\in\im(D-I),$ where $Z=(\w{B}_{\inv a})_a$ and ${V}=(B_a)_a.$

 We now prove a) implies b).  Assume $d\geq 3.$

Equality ${\mathcal D} P=P$ is verified by at least $2$ linear independent vectors not proportional to $u.$
They verify
\begin{equation}\label{eq1}\left\{\begin{array}{l}
(D-I) P_1+A P_2+B P_3+ C P_4=0\\
(D-I) P_2+ B P_4=0 \\
(D-I) P_3+ A P_4=0 \\
(D-I) P_4=0
\end{array}\right.\end{equation}

The last equation in \eqref{eq1} implies $P_4$ is either zero or proportional to $Z$.

If $P_4=cZ$, without loss of generality we can assume $c=1.$ From the third equation in \eqref{eq1}
we get  $AZ\in\im(D-I)$ which is equivalent, by \eqref{imk}, to b).

If $P_4=0,$ the second and third equation in \eqref{eq1} imply both $P_2$ and $P_3$ are either zero
or proportional to $Z,$ but in any case $(P_2,P_3)\neq (0,0).$

 If $(P_2,P_3)= (0,Z), (Z,0),$ from the first equation in \eqref{eq1}, we get,  passing to adjoint
if needed, that $AZ\in\im(D-I),$ which is equivalent, by \eqref{imk}, to b).

Finally, assume each solution of \eqref{eq1} verifies $(P_2,P_3, P_4)= (c Z, Z, 0)$ for some
$0\neq c\in\C,$ (we can always assume one of the constants is equal to $1$).

Since $d\geq 3,$ there are at least two such solutions, say $P$ and $P'.$ With obvious meaning
of symbols, by the first equation in \eqref{eq1}
$$(D-I)(P_1'-P_1)+(c'-c) A Z=0.$$
If $c'\neq c,$ the latter implies b).

On the other hand $c'=c$ is impossible. On the contrary, if $c'=c$ we get $P'_1=P_1+\alpha Z,$ $\alpha\in\C.$
In other words, up to constant, there is only one other vector not proportional to $u$  in
a basis of  $\Ker(\mcD-I),$ in contradiction with $d\geq 3.$

We now show that b) implies a).

As already noted, b) means that $AZ\in\im(D-I),$ and passing to adjoint,
$BZ\in\im(D-I),$ too.
Therefore there are  vectors $W_2=(W_{2,a})_a,$ and $W_3=(W_{3,a})_a,$ such that
$W_{2,a}=W_{3,a}^{\ast},$ and
\begin{equation}\label{eq2}
(D-I) W_2+ B Z=0, \quad
(D-I) W_3+ A Z=0.
\end{equation}

It follows that
$(W_3, Z, 0, 0)^{\top}$ and $(W_2, 0, Z, 0)^{\top}$
are in $\Ker(\mcD-I),$ and together with $u$ they form a set of three linear independent vectors.
Hence $\dim\Ker(\mcD-I)\geq 3.$
\end{proof}

\begin{lem}\label{lemq4}
Assume that the dimension of
the eigenspace corresponding to the eigen\-va\-lue $1$  of  $\mcD$ is $d\geq 3.$

Then  there exist linear maps
$Q_b:V_b\rightarrow \w{V}_b,$  $b\in A,$ such that
 the vector
$$\left(\begin{array}{c}
(\w{B}_{b^{-1}}\, Q_b^{\ast})_{b}\\
( Q_b \, \w{B}_{b^{-1}})_{b}\\
(\w{B}_{b^{-1}})_{b}
\end{array} \right)$$
is  a (right) eigenvector corresponding to the eigenvalue
$1$ of the the principal submatrix
$\mathcal{D}_1$ of $\mathcal{D},$
 obtained by deleting the rows and columns of
$D_{1,1}.$
\end{lem}
\begin{proof}
  Following the notation of Lemma \ref{dim>3}, set $S_2=\text{diag}(\Id\otimes \overline{K_a},\,a\in A),$
 and $S_2=\text{diag}(K_a\otimes \Id,\,a\in A).$

Since $d\geq 3,$  from the proof of Lemma
\ref{dim>3}, we get that
there are  vectors $W_2=(W_{2,a})_a,$ and $W_3=(W_{3,a})_a,$ such that
$W_{2,a}=W_{3,a}^{\ast}$ and
$$\left\{\begin{array}{l}
(D_{22}-I) S_2 W_2+ D_{24} Z=0 \\
(D_{33}-I) S_3 W_3+ D_{34} Z=0 \\
(D_{44}-I) Z=0,
\end{array}\right.
$$
where $Z$ is an eigenvector corresponding to the eigenvalue $1$ for $D_{44}$, and, without
loss of generality we can assume
$Z=U_4=(\w{B}_{\inv a})_a,$ see Lemma \ref{autodim4}.
Also note that
$$(S_3 W_3)_b=K_b W_{3,b}:\w{V}_{\inv b}=V_{b}^{\ast}\rightarrow \w{V}_b. $$

Since $\w{B}_{\inv b}$ is strictly positive definite,  it is invertible as
 a linear map $\w{B}_{\inv b}:V_{b}^{\ast}\rightarrow V_{ b}.$  Hence the map
$$Q_b=K_b\,W_{3,b}\,\w{B}_{\inv b}^{-1}:V_b\rightarrow\w{V}_b,$$
is well defined and linear. It follows $K_b\,W_{3,b}=Q_b\,\w{B}_{\inv b}$
and $$(S_2 W_2)_b= W_{2,b} K_b^{\ast}=(K_b\, W_{3,b})^{\ast}=\w{B}_{\inv b}\,Q_b^{\ast}.$$
\end{proof}
\begin{thm}
Under the hypotheses hitherto assumed on  the matrix systems, the dimension of the eigenspace
corresponding to the eigen\-va\-lue $1$  of  $\mcD$  is $d=4$ if and only if the following conditions hold
\begin{itemize}
\item[1)] $\displaystyle{\sum_{a,b\in A}{\tr(\w{H}_{ab}\, B_{\inv b}\, K_{\inv b}^{-1} \,E_{a b}^{\ast}\,\w{B}_a)}}=0;$ 
\item[2)] Given  the linear maps $Q_b:V_b\rightarrow \w{V}_b,$ provided by
 Lemma \ref{lemq4}, the following identity holds:
 $$\w{H}_{ab}Q_b+ E_{ab}=Q_a H_{ab},\quad \text{for all}\; {a,b\in A}.$$
\end{itemize}
\end{thm}
\begin{proof}
Condition 1) is equivalent to $d\geq 3$ by Lemma \ref{dim>3}.
Condition 2) is equivalent, by the same argument used in \cite[Theorem 5.13]{KSS16}
 to
 \begin{eqnarray*}
0\!\!\!&=&\!\!\!\!\sum_{a,b\in A}{\mathrm{tr}}{
(\w{B}_a \, E_{ab}\, \w{B}_{\inv b}\, Q_b^{\ast}\,\w{H}_{ab}^{\ast}
+\w{B}_a\, \w{H}_{ab}\, Q_b\, \w{B}_{\inv b}\, {E}_{ab}^{\ast}
+\w{B}_a\, {E}_{ab}\, \w{B}_{\inv b}\,{E}_{ab}^{\ast})}\\
&=&\!\!\!\mathrm{tr}\left(\sum_{a\in A}{\w{B}_a
\sum_{b\in A}{\left[ E_{ab}\w{B}_{\inv b}Q_b^{\ast}\w{H}_{ab}^{\ast}+
\w{H}_{ab}Q_b \w{B}_{\inv b}{E}_{ab}^{\ast}+
 E_{ab}\w{B}_{\inv b}{E}_{ab}^{\ast}\right]}}\right),
\end{eqnarray*}
meaning that the vector
$$W=(W_j)_{j=2,3,4}=\left(\begin{array}{c}
(\w{B}_{b^{-1}}\, Q_b^{\ast})_{b}\\
( Q_b \, \w{B}_{b^{-1}})_{b}\\
(\w{B}_{b^{-1}})_{b}
\end{array} \right)$$
is such that
$D_{1,2} W_2+ D_{1,3} W_3+ D_{1,4} W_4\in\im(D_{11}-I).$

Hence conditions 1) and 2) are equivalent to  the existence of a fourth linear independent vector
$W=(W_j)_{j=1,\dots,4}$  in (the basis of) $\Ker(\mcD-I)$  besides those provided in the proof of
Lemma \ref{dim>3}.
\end{proof}
%

We summarize  the results for equivalent systems in the following Theorem:
\begin{thm}
 If the two systems  $(V_a,H_{ba},B_a)$ and
 $(\widehat{V}_a,\widehat{H}_{ba},\widehat{B}_a)$
  {are equivalent}, then $1$ is an eigenvalue of multiplicity four for $\mcD$ and the dimension
of the eigenspace of $1$ is two, three or four.
  In particular
  \begin{itemize}
\item[\textbf{BI}] The dimension of the eigenspace of $1$
  is $4$ if and only if, for every $a,b,\in A$ and $v_a,v_b\in V_a,V_b$, one has
  \begin{equation}\label{crescita1/eps}
 \|\phi_\eps^{v_a,v_b}\|_2^2
\simeq\frac1\eps,\quad \text{as}\quad \eps\rightarrow 0.
\end{equation}

As well as in the case \textbf{AI} there exists a tuple of linear maps
$Q_a:V_a\to\widehat{V}_a$ satisfying \eqref{imain}
\begin{equation*}
  {\widehat{H}}_{ab}Q_b+ E_{ab}=Q_a H_{ab},\;  {a,b\in A},
\end{equation*}
but in this case the tuple is not unique.
\item[\textbf{BII}]
  The dimension of the eigenspace of $1$ is $2$ (respectively $3$)
  if and only if
\begin{equation*}
\|\phi_\eps^{v_a,v_b}\|_2^2 \simeq\frac1{\eps^3}\quad (\text{respectively}\; \frac1{\eps^2}),
\quad \text{as}\quad \eps\rightarrow 0.
\end{equation*}
%
In both cases  {no vector in $\mcH$}
satisfies the good vector bound \eqref{GVBo}
and no system of $Q_a:V_a\to\widehat{V}_a$ can satisfy \eqref{imain}.
  \end{itemize}

\end{thm}
\begin{proof} The only thing to prove is that the dimension
  of the eigenspace of $1$ cannot be $1$. If this were the case one should have a cycle
of length four in the generalized eigenspace of $1$. This implies that the growth of the quantity is
$$
 \sum_{|x|=n}
{|\langle \mu[e,a,v_a],\pi(x)\mu[e,b,v_b]\rangle  |^2}
e^{-\eps|x|}
\simeq n^3
$$
which contradicts Haagerup's inequality~\eqref{haagerup}.
\end{proof}

In \cite{KSS16}  we proved
that all the representations arising from \textbf{AII}  satisfy {monotony}.
The calculation and the techniques  presented in \cite{KSS16} can be carried
out virtually unchanged to prove that
also the representations arising from \textbf{BII} satisfy {monotony}.

This paper is devoted to the study of cases \textbf{AI, BI}.

\section{Equivalent  multiplicative $\Gamma$-representations}\label{equivG}

In this section the reigning hypotheses are as follows:

$\bullet$\;\; $\pi$ and
$\widehat{\pi}$ are multiplicative representations  built up from
  irreducible twin systems $(V_a,H_{b a},B_a)$
  and $(\widehat{V}_a,\widehat{H}_{b a},\widehat{B}_a)$: this implies that $\pi$ and
  $\widehat{\pi}$ are irreducible as representations of $\cp$
  \cite[Theorem 5.3]{KS04};

$\bullet$\;\;   For every $a,b,\in A$ and $v_a,v_b\in V_a,V_b$, one has \eqref{crescita1/eps}
and hence there exists a tuple of linear maps
$Q_a:V_a\to\widehat{V}_a$ satisfying:
\begin{equation}\label{main}
  {\widehat{H}}_{ab}Q_b+ E_{ab}=Q_a H_{ab},\;  {a,b\in A}.
\end{equation}

We begin with the following
\begin{lem}\label{qq}
  The maps $Q_a:V_a\to \widehat{V}_a$ appearing in  \eqref{main} also  satisfy
$$
  Q_a^{\ast}+Q_{a^{-1}}=0,\quad a\in A
$$
\end{lem}
\begin{proof}
  Take the adjoint in
${\widehat{H}}_{ab}Q_b+ E_{ab}=Q_a H_{ab},$
 and write $\inv b$ for $a$ and $\inv a$ for $b$:
$$Q_{a^{-1}}^{\ast}{\widehat{H}}_{{b^{-1}}{a^{-1}}}^{\ast}+
 E_{{b^{-1}}{a^{-1}}}^{\ast}=H_{{b^{-1}}{a^{-1}}}^{\ast}Q_{b^{-1}}^{\ast}.$$
 Remember that $E^*_{\inv b\inv a}=E_{ab}$ and write
   \begin{align*}
& \widehat{H}_{ab}Q_{b^{-1}}^{\ast}=Q_{a^{-1}}^{\ast}{{H}}_{a b}+
E_{a b}\\
&{\widehat{H}}_{ab}Q_b+ E_{ab}=Q_a H_{ab}
\end{align*}

Adding up to the two  equations gives
$${\widehat{H}}_{ab}(Q_b+Q_{b^{-1}}^{\ast})=
(Q_a+Q_{a^{-1}}^{\ast}) H_{ab},\;  {a,b\in A}.$$
If the two systems are inequivalent $(Q_a+Q_{a^{-1}}^{\ast})$
  must be zero
  by Remark 3.4 in \cite{KS04}.
  If the two systems are equivalent one can replace
  $Q_a$ by $\tilde{Q}_a=\frac12(Q_a-Q_{\inv a}^*)$.
\end{proof}
%
%
%
%

\begin{thm}\label{equi}
  Assume the existence of
  a tuple of linear maps
$Q_a:V_a\to\widehat{V}_a$ satisfying \eqref{main}.

Then there exists a linear bijection $J:\mcH^\infty\to\w{\mcH}^\infty$
that intertwines
 $\pi$ to $\widehat{\pi}$.
\end{thm}
\begin{proof}

  We shall first define $J$
  for functions $\mu[x,xa,v_a]$ with $|xa|\!=\!|x|+1$.

Recall that the tree decomposes as the disjoint union
of the sets $\G(x,xa)$ and $\G(xa,x)$.
Write $\mu[xa,x, v_{\inv a}]$
for $\mu[xa,xa\inv a, v_{\inv a}]$
 and observe that the maps $\mu[x,xa, v_{ a}]$ and $\mu[xa,x, v_{\inv a}]$ are
  orthogonal.

  For all $a\in A$ and $v_a\in V_a$ define
\begin{equation}\label{J}
J(\mu[x,xa,v_a])=\widehat{\mu}[x,xa,-Q_av_a]+\widehat{\mu}[x a,x,B_av_a],
\end{equation}
where $\widehat{\mu}$ is the multiplicative function constructed
from the system $(\w{V}_a,\w{H}_{ab},\w{B}_a)$ according to \eqref{mu}.
Note that  $Q_av_a\in\widehat{V}_a$, $B_av_a\in V_a^{\ast}=\widehat{V}_{a^{-1}}$.

Since
$\pi(y)\mu[x,xa,v_a]=\mu[yx,yxa,v_a]$  regardless of whether $yx$ or $xa$
are reduced or not, it is clear that
 $J$ will intertwine $\pi$ to $\w{\pi}$.

%
To see that
$J$ is well defined write
$$
\mu[e,a,v_a]=\sum_{b\neq\inv a}\mu[a,ab,H_{ba}v_a]
$$
and compute
$$J(\sum_{b\neq\inv a}\mu[a,ab,H_{ba}v_a])=
  \sum_{b\neq a^{-1}}{
\left(\widehat{\mu}[a,ab,-Q_b H_{ba}v_a]+
\widehat{\mu}[ab,a,B_b H_{ba}v_a]\right)}.
$$
We need
\begin{equation}\label{intrallaccia1}
  J(\mu[e,a,v_a])=\widehat{\mu}[e,a,-Q_av_a]+\widehat{\mu}[a,e,B_av_a]
\end{equation}
\begin{equation}\label{intrallaccia2}
 = \sum_{b\neq a^{-1}}{
\left(\widehat{\mu}[a,ab,-Q_b H_{ba}v_a]+
\widehat{\mu}[ab,a,B_b H_{ba}v_a]\right)}\,.
\end{equation}

Compute \eqref{intrallaccia1} at $x=e$ to get
$B_av_a$, while \eqref{intrallaccia2} evaluated at $e$ gives
\begin{equation*}
  \sum_{b\neq\inv a} \widehat{H}_{\inv a\inv b}B_bH_{ba}v_a=
    \sum_{b\neq\inv a} H^*_{ba}B_bH_{ba} v_a,
\end{equation*}
they are equal since $(B_a)$ is a right eigenvector
of the matrix $(H^*_{ba}~\otimes~H'_{ba})_{a,b}$.
Proceed now to compute \eqref{intrallaccia1} at $x=ac$: we get
 $-\widehat{H}_{ca}Q_av_a$ while \eqref{intrallaccia2}  gives
$$
-Q_c H_{ca}v_a+\sum_{b\neq c,\inv a}\w{H}_{c\inv b}B_bH_{ba}v_a\,.
$$
Equality holds if and only if
\begin{equation*}
     Q_c H_{ca}=\Hcap_{ca}Q_a\;+\;\sum_{b\neq c, \inv a}\w{H}_{c\inv b}B_bH_{ba},
\end{equation*}
which is \eqref{main} where we wrote $a$ for $c$, $b$ for $a$ and $c$ for $b$.

We shall now provide the expression for $J(\mu[a,e,v_{\inv a}])$.
Observe that
\begin{eqnarray*}
J(\pi(\inv a)\mu[e,a,v_a])&=&\widehat{\pi}(\inv a)J\mu[e,a,v_a]\\
  &=&\widehat{\mu}[\inv a, e, -Q_av_a]+\widehat{\mu}[e,\inv a, B_av_a],
\end{eqnarray*}
write now $a$ for $\inv a$: we get
\begin{equation*}
  J\mu[a,e,v_{\inv a}]=\widehat{\mu}[e,a, B_{\inv a}v_{\inv a}]+\widehat{\mu}[a, e, -Q_{\inv a}v_{\inv a}].
\end{equation*}

Identify $V_a\oplus V_{\inv a}$ with the space of multiplicative functions of the form
$\mu[e,a,v_a]+\mu[\inv a,e, v_{\inv a}].$

With this identification  we may think of $J$ as acting
on each $V_a\oplus V_{\inv a}$ via the matrix
\begin{equation}\label{matJ}
\left(
\begin{array}{cc}
-Q_a &B_{a^{-1}}\\
B_a & -Q_{a^{-1}}
    \end{array}
    \right)
    :V_a\oplus V_{a^{-1}}\rightarrow \w{V}_a\oplus \w{V}_{a^{-1}}\,,
    \end{equation}
where $\w{V}_a\oplus \w{V}_{a^{-1}}$ is identified with the space of functions
of the form $\w{\mu}[e,a,v_a]+\w{\mu}[\inv a,e, v_{\inv a}]$.
Moreover, the matrix \eqref{matJ}
has as right inverse

\begin{equation}\label{invmatJ}
\left(\!\!\!
\begin{array}{cc}
-B_a^{-1}Q_a^{\ast}(B_{a^{-1}} +Q_{a}B_{a}^{-1}Q_{a}^{\ast})^{-1}\!&\!\!(B_a +Q_{a^{-1}}B_{a^{-1}}^{-1}Q_{a^{-1}}^{\ast})^{-1}\\ \\
(B_{a^{-1}} +Q_{a}B_{a}^{-1}Q_{a}^{\ast})^{-1}\!&\!\!-B_{a^{-1}}^{-1}Q_{a^{-1}}^{\ast}
(B_{a}+Q_{a^{-1}}B_{a^{-1}}^{-1}Q_{a^{-1}}^{\ast})^{-1}
    \end{array}\!\!\!
    \right)\!.
    \end{equation}
\vspace{1.5mm}

Indeed, a simple multiplication and Lemma \ref{qq} show that it suffices
    to prove that the  matrix \eqref{invmatJ} is well defined, i.e.
    $B_a+Q_{\inv a} B_{\inv a}^{-1} Q_{\inv a}^{\ast}$ is
    invertible. But the latter is positive definite, hence invertible.

Extend now $J$ by linearity to
  $\mcH^{\infty}$ by \eqref{effe}.

    We shall now prove
that $J$ is invertible from $\mcH^\infty\to\w{\mcH}^\infty$.


To prove  
surjectivity is sufficient,
by linearity and the intertwining property, to check that all the functions
of the type $\w{\mu}[e,a,\hat{v}_a]+\w{\mu}[a,e,\hat{v}_{\inv a}]$ are in the image of $J$.

By \eqref{invmatJ} we know that, given
$\hat{v}_a\in\w{V}_a$ and $\hat{v}_{\inv a}\in\w{V}_{\inv a}$,
there exist $w_a\in V_a ,$ $w_{\inv a}\in V_{\inv a}$ such that
$$\left(
\begin{array}{cc}
-Q_a &B_{a^{-1}}\\
B_a & -Q_{a^{-1}}
    \end{array}
    \right)
     \left(
\begin{array}{c}
w_a \\
w_{\inv a}
    \end{array}
    \right)
    = \left(\begin{array}{c}
\hat{v}_a \\
\hat{v}_{\inv a}
\end{array}
    \right),$$
and so
\begin{equation*}
J(\mu[e,a,w_a]+\mu[a,e,w_{\inv a}])=\widehat{\mu}[e,a,\hat{v}_a]+\widehat{\mu}[a,e,\hat{v}_{\inv a}]\;.
 \end{equation*}

We proof that $J$ is injective as follows. For any $n\in\N$ let $W_n$ be the finite dimensional
subspace of $\mcH^{\infty}$ consisting of multiplicative functions of the form
$$
f=\sum_{|x|=n}\sum_{\atopn{a\in A}{|xa|=|x|+1}}
{\mu[x,x a,f(xa)]}.
$$
Let $J_n$ be the
restriction of $J$ on this subspace. It is easy to see that $J_n:W_n
\rightarrow \w{W}_n$ is onto, therefore $J_n$ is also
one-to-one. Now, let $J f=0$ for some multiplicative
function $f$. By \eqref{effe} there exists
 $n$  such that  $f\in W_n$ and $J_n f=0,$ therefore $f=0$ and the theorem is proved.
\end{proof}
%
%
%
Incidentally we have proved the following
\begin{prop}\label{matricegenerale}
Any intertwiner $T$ between $\pi$ and $\picap$
which sends functions of the form
$\mu[e,a,v_a]+\mu[\inv a,e, v_{\inv a}]$ into functions of the same type lying in $\w{\mcH}^\infty,$
must be given by a matrix
\begin{equation*}
\left(
\begin{array}{cc}
X_a &Y_{a^{-1}}\\
Y_a & X_{a^{-1}}
    \end{array}
\right)
\end{equation*}
where  $X_a:V_a\rightarrow \w{V}_a,$
$Y_a:V_a\rightarrow \w{V}_{\inv a},$ must be chosen so that
\begin{eqnarray*}
 & Y_a=\lambda B_a \quad \mbox{for some  $\lambda\in\C,$}\\
  & X_a H_{ab}+\lambda E_{ab}=\w{H}_{ab}X_b\,.
  \end{eqnarray*}

In particular the maps $-X_a$ must satisfy \eqref{main} with respect to the
choice of  $\lambda B_a$ .
\end{prop}
\begin{proof}
Assume that $J(\mu[e,a,v_a])=\widehat{\mu}[e,a,X_av_a]+\widehat{\mu}[a,e,Y_av_a]$
and   proceed as in the proof of Theorem \ref{equi} to see that
\begin{equation*}
  J(\mu[e,a,v_a])=\widehat{\mu}[e,a,X_av_a]+\widehat{\mu}[a,e,Y_av_a]
  \end{equation*}
    {must be equal to}
    \begin{equation*}
  \sum_{b\neq a^{-1}}{
\left(\widehat{\mu}[a,ab,X_b H_{ba}v_a]+
\widehat{\mu}[ab,a,Y_b H_{ba}v_a]\right)}\;.
\end{equation*}
The two expressions will be equal if and only if
$Y_a$ is a multiple of $B_a$, say $\lambda B_a$, and
the maps $-X_a$  satisfy \eqref{main} with respect to the
choice of  $\lambda B_a$.

If the two systems   $(V_a,H_{b a},B_a)$
  and $(\widehat{V}_a,\widehat{H}_{b a},\widehat{B}_a)$
 {are not}
 equivalent then
  there is
 only one possibility for the maps $Q_a$ and hence the matrix for $T$ will be
\begin{equation*}
\left(
\begin{array}{cc}
-\lambda Q_a &\lambda B_{a^{-1}}\\
\lambda B_a & -\lambda Q_a^{\ast}\\
\end{array}    \right).
\end{equation*}


If the two systems   $(V_a,H_{b a},B_a)$
and $(\widehat{V}_a,\widehat{H}_{b a},\widehat{B}_a)$ \textbf{are } equivalent
then there exists a tuple of linear maps $K_a:V_a\to\w{V}_a$ so that
\begin{equation*}
  K_a H_{ab}=\w{H}_{ab}K_b
\end{equation*}
and this will increase the possibilities for the $Q_a$. Having in mind that
we want to keep the condition $Q_a^\ast + Q_{\inv a}=0$, observe that
we may always assume that $K_a^\ast=K_{\inv a}$ (see Remark \ref{Kauto}).

In this case, for a fixed given choice of  $\lambda B_a$
all the possible choices for $X_a$ satisfying $X_a^\ast +X_{\inv a}=0$
are of the form
$$
X_a=-\lambda Q_a +i c K_a\quad\mbox{for real $c$},
$$
and the matrix for $T$ will be
\begin{equation*}
\left(
\begin{array}{cc}
-\lambda Q_a +icK_a &\lambda B_{a^{-1}}\\
\lambda B_a & -\lambda Q_a^{\ast}+icK^*_a\\
\end{array}    \right).
\end{equation*}
\end{proof}

We are now ready to prove the following
\begin{thm}
  Assume the existence of
  a tuple of linear maps
$Q_a:V_a\to\widehat{V}_a$ satisfying \eqref{main}.

Let $J$ be as in~\eqref{J}. Then $J:\mcH^\infty\to\w{\mcH}^\infty$
preserves the inner product and hence
it extends to a unitary equivalence between $\pi$ and $\w\pi$.
\end{thm}

\begin{proof}
  Let $Q_a$ be a tuple of maps satisfying \eqref{main}.
  Construct $J$ as in Theorem \ref{equi} and let
  \begin{equation*} 
\left(
\begin{array}{cc}
-Q_a &B_{a^{-1}}\\
B_a & -Q_{a^{-1}}
    \end{array}
    \right)
\end{equation*}
be the matrix representing it.

At the algebraic level we know that
 $J$ is invertible and intertwines $\pi$ to
$\widehat{\pi},$ from $\mcH^\infty$ to $\w\mcH^\infty$. Since
the twin of the twin system is the original system and
$J^{-1}:\w{\mcH}^{\infty}\rightarrow {\mcH}^{\infty}$ intertwines
the representations  as well, Proposition~\ref{matricegenerale}
applied to the twin system, says that
the matrix \eqref{invmatJ} for $J^{-1}$ must be of the form
\begin{equation}\label{invJ}
\left(
\begin{array}{cc}
-\w{Q}_a &\w{B}_{a^{-1}}\\
\w{B}_a & -\w{Q}_{a^{-1}}
    \end{array}
\right)
\end{equation}
where $\w{B}_a$ is a right eigenvector of the matrix
$(\w{H}^*_{ba}~\otimes~\w{H}'_{ba})_{a,b}$ and the maps
$\w{Q}_a:\w{V}_a\to {V}_a$ satisfy
\begin{equation}\label{main2}
  {{H}}_{ab}\w{Q}_b+ \w{E}_{ab}=\w{Q}_a \w{H}_{ab},\quad  {a,b\in A},
\end{equation}
with $\w{E}_{ab}$ the analogue of $E_{ab}$ with respect to the twin system.

Hence
\begin{equation*}
J^{-1}(\widehat{\mu}[e,a,\hat{v}_{a}])=
{\mu}[e,a,-\w{Q}_a \hat{v}_a]+{\mu}[a,e,\w{B}_a \hat{v}_a].
\end{equation*}

Remember that \eqref{invJ} is also a left inverse for the matrix
\eqref{matJ} and use
Lemma \ref{qq}  to get
\begin{align}\label{serviranno}
  &\w{Q}_aQ_a+\w{B}_{\inv a}B_a=\operatorname{Id}\\
  &\w{B}_a Q_a=-\w{Q}_{\inv a}B_a \label{serviranno3}\\
  & Q_{\inv a}\w{B}_a=-B_a\w{Q}_a\label{serviranno4}
  \end{align}
Now we prove that $J$ preserves the inner product in $\mcH^{\infty}$ defined in  \eqref{inner}.

Due to the structure of multiplicative functions,
by linearity and sesquilinearity, since $\pi$ and $\w{\pi}$ are unitary representations intertwined
by $J$ and $J^{-1},$
it is sufficient to show that
\begin{equation*}
  \langle Jf,J g\rangle=\langle f, g\rangle
\end{equation*}
when either $f=\mu[x,xa,v_a]$ and $g=\mu[x,xa,w_a]$ or
$f=\mu[a_1,e,w_{\inv {a_1}}]$ and $g=\mu[x,xa,v_a],$  where $a_1$ is the first
letter of $x,$ and $xa$ is reduced.
In the first case one has
\begin{equation}\label{primoa}
  \langle \mu[x,xa,v_a],\mu[x,xa,w_a]\rangle
  =B_a(v_a,w_a)
\end{equation}
while
\begin{eqnarray}\label{secondoa}
&& \langle J\mu[x,xa,v_a],J\mu[x,xa,w_a]\rangle \\
&=&
\langle
\widehat{\mu}[e,a,-Q_a v_a]+ \widehat{\mu}[a,e,B_a v_a],
\widehat{\mu}[e,a,-Q_a w_a]+ \widehat{\mu}[a,e,B_a w_a]
\rangle\nonumber\\
&=&
\w{B}_{\inv a}(B_av_a,B_a w_a)+
\w{B}_a (Q_a v_a,Q_a w_a)\;.\nonumber
\end{eqnarray}
Hence \eqref{primoa} will be equal to \eqref{secondoa} if and only if
$$
B_a = B_a^*\w{B}_{\inv a}B_a + Q_a^*\w{B}_aQ_a
= B_a\w{B}_{\inv a}B_a + {B}_a\w{Q}_aQ_a,
$$
(after an application of \eqref{serviranno3} and Lemma \ref{qq}), which is true by  \eqref{serviranno}.

In the second case, since $ \mu[a_1,e,w_{\inv a_1}]$ and
$\mu[x,xa,v_a]$ are orthogonal,
we need to prove $\langle J f,J g\rangle=0.$
Let us suppose that $xa=a_1\dots a_{n}a_{n+1}$  is reduced.
  We have, again by orthogonality
and since $a_1^{-1}xa=a_2\dots \! a_{n+1},$
\begin{eqnarray*}
\langle J f,J g\rangle\!\!\!\!&=&\!\!\!\langle J \mu[a_1,e,w_{\inv a_1}],J \mu[x,xa,v_a]\rangle\\
%
%
&=&\!\!\!
\langle  \widehat{\mu}[e,\inv a_1,-Q_{\inv a_1}w_{\inv a_1}],\w{\pi}(a_2\dots a_n a_{n+1})\,
\widehat{\mu}[e,\inv a_{n+1},B_{a_{n+1}} v_{a_{n+1}}]\rangle\\
& &\!\!\!+
\langle \widehat{\mu}[e,a_1,B_{\inv a_1} w_{\inv a_1}],\w{\pi}(a_1\dots a_n) \,\widehat{\mu}[e,a_{n+1},-Q_{a_{n+1}} v_{a_{n+1}}]\rangle \\
& &\!\!\!
+
\langle \widehat{\mu}[e,a_1,B_{\inv a_1} w_{\inv a_1}], \w{\pi}(a_1\dots a_n a_{n+1})\,\widehat{\mu}[e,\inv a_{n+1},B_{a_{n+1}} v_{a_{n+1}}]\rangle.
\end{eqnarray*}
 Each of the above quantities have been already calculated in  the proof of \cite[Lemma 5.5]{KSS16} and are respectively equal to
\begin{eqnarray*}
& &\!\!\!\!\!\!
-\widehat{B}_{\inv a_1}(Q_{\inv a_1} w_{\inv a_1}\,,\,\widehat{H}_{{\inv a_1} {\inv a_2}}\dots
\widehat{H}_{{\inv a_n} {\inv a_{n+1}}}B_{a_{n+1}} v_{a_{n+1}})\\
& &\!\!\!\!\!\!
-\widehat{B}_{a_{n+1}}(\widehat{H}_{a_{n+1} a_n}\dots
\widehat{H}_{a_2 a_1} B_{\inv a_1} w_{\inv a_1}\,,\, Q_{a_{n+1}} v_{a_{n+1}})\\
& &\!\!\!\!\!\!
+\!\sum_{j=0}^{n-1}\!\widehat{E}_{a_{j+2} a_{j+1}}\!(\widehat{H}_{a_{j+1} a_{j}}\!\dots\!
\widehat{H}_{a_{2}a_{1}}B_{a_1^{-1}} w_{a_1^{-1}},
\widehat{H}_{a_{j+2}^{-1} a_{j+3}^{-1}}\!\dots\!\widehat{H}_{a_{n}^{-1} a_{n+1}^{-1}}B_{a_{n+1}} v_{a_{n+1}}).
\end{eqnarray*}
Therefore, after an application of \eqref{serviranno3}, \eqref{serviranno4}, and Lemma \ref{qq},
everything is proved if we show that, for every $n\geq 1,$
\begin{eqnarray}
& &
H_{a_{n+1} a_n} \dots H_{a_2 a_1} \w{Q}_{ a_1}
-\w{Q}_{a_{n+1}} \w{H}_{a_{n+1} an} \dots \w{H}_{a_2 a_1} \label{fin} \\
& &
+\sum_{j=0}^{n-1}H_{a_{n+1} a_n}\dots  H_{a_{j+3} a_{j+2}}\w{E}_{a_{j+2} a_{j+1}}\w{H}_{a_{j+1} a_j} \dots \w{H}_{a_2 a_{1}}=0.
\nonumber
\end{eqnarray}

The latter can be easily proved by induction on $n\geq 1,$  by means of \eqref{main2}.
 We omit the details, just note that
if $n=1,$  \eqref{fin} is indeed
$${H}_{a_2 a_1}\widehat{Q}_{a_1}
- \widehat{Q}_{a_2}\widehat{H}_{a_2 a_1}
+\widehat{E}_{a_2 a_1}=0.$$
\end{proof}
\begin{thm}
  Assume the reigning hypotheses  of this Section.
  If the two systems $(V_a,H_{b a},B_a)$
  and $(\widehat{V}_a,\widehat{H}_{b a},\widehat{B}_a)$ are  equivalent
  then $\pi$ splits into the sum of two   $\G$-representations.
\end{thm}
\begin{rem}
  It will be proved in the next Section that these two $\G$
  representations are indeed
  irreducible and inequivalent.
\end{rem}
\begin{proof}
  Let $(K_a)$ be an equivalence between the two systems. By
Lemma 5.2 of  \cite{KS04} we may assume that
$K_a:V_a\rightarrow\w{V}_a$ is a $A$-tuple of unitary operators satisfying, by
Remark \ref{Kauto}, $K_{\inv a}=K^\ast_a$.
Fix 
a system of $Q_a$ satisfying \eqref{main}. Normalize
the $\w{B}_a$ so that
\begin{equation*}
 J= \left(
 \begin{array}{cc}
   -Q_a & B_{\inv a}\\
   B_a & -Q_a^\ast\\
   \end{array}
 \right)
 \quad
 \inv J=
 \left(
 \begin{array}{cc}
   -\w{Q}_a & \w{B}_{\inv a}\\
  \w{B}_a & -\w{Q}_a^\ast\\
   \end{array}
 \right)
\end{equation*}
and set
 \begin{equation*}
 \mcK=\left(
 \begin{array}{cc}
   K_a & 0\\
   0 & K_{\inv a}\\
   \end{array}\right)\,.
\end{equation*}
 Consider the operator $\mcK \inv J\mcK:\mcH^\infty\rightarrow\w{\mcH}^\infty$.
We may assume that $J$ is unitary as well, so that
$\mcK \inv J\mcK$ is also unitary and, by Theorem \ref{equi}, intertwines the two $\G$-representations
$\pi$ and $\w\pi$.
By Proposition \ref{matricegenerale} it must be of the form
\begin{equation*}
\left(
\begin{array}{cc}
-\lambda Q_a +icK_a &\lambda B_{a^{-1}}\\
\lambda B_a & -\lambda Q_a^{\ast}+icK^*_a\\
\end{array}    \right)\,.
\end{equation*}

We pass to calculate the product of the three matrices. Since  the
term $K_{\inv a}\w{B}_aK_a=K^\ast_a\w{B}_aK_a$ is positive, the same must be true for $\lambda B_a$, so that
$\lambda$ must be positive.
Hence $\mcK \inv J\mcK$ must satisfy the following equation:
\begin{equation*}
  \mcK \inv J\mcK=\lambda J + ic\mcK
\end{equation*}
for some positive $\lambda$ and real $c$. Multiply both sides by $\inv\mcK
J\inv\mcK$ to get the equation:
$$
\operatorname{Id}=\lambda (\inv\mcK J)^2+ic(\inv\mcK J).
$$
Hence the unitary operator $(\inv\mcK J)$
has two complex eigenvalues, say $\lambda_\pm,$ which
are distinct since $(\inv\mcK J)$ is not a scalar. Some elementary algebra
shows that $\lambda$ must be $1$ and $\lambda_\pm= \frac{-ic\pm\sqrt{4-c^2}}2$.
Hence the $\G$-representation splits into the direct sum of two representations
each corresponding to an eigenspace of $(\inv\mcK J)$.
\end{proof}
\begin{rem}\label{autovalori}
  We remark that if we replace $J$ by
  $$
  \tilde{J}= \frac2{\sqrt{4-c^2}}(J-\frac{ic}2\mcK)
  $$
  we still obtain a valid non trivial intertwiner for $\pi$ and $\w\pi$, but
  this choice of $\tilde{J}$ will lead to the unitary operator
  \begin{equation*}
\mcJ=(\inv\mcK )\tilde{J}
    \end{equation*}
  having the simpler eigenvalues $+1$ and $-1$.
\end{rem}

\section{$\G$-irreducibility of multiplicative representations}
We begin this section by recalling the Duplicity and the Oddity Theorems
from \cite{HKS19}. Denote by $\norm{\cdot}_{HS}$ the Hilbert-Schmidt norm
of an operator.
\begin{dthm}\label{DupTheorem} Let $\pi:\Gamma\to\mcU(\mcH)$ be a unitary
  representation of~$\Gamma$. Suppose
  \begin{itemize}
  \item[$\bullet$] $(\pi'_\pm,\mcH'_\pm)$ are two irreducible
    $\cp$-representations, inequivalent as $\cp$-representations.
  \item[$\bullet$] $\iota_\pm:\mcH\to\mcH'_\pm$ are two perfect boundary
    realizations of~$\pi$.
  \item[$\bullet$] The following Finite Trace Condition  holds
    \begin{equation}\tag{FTC}\label{FTC}
      \norm{(\iota_-^* \pi'_-(\one_{a})\iota_-)
      (\iota_+^* \pi_+'(\one_{b})\iota_+)}_{HS}<\infty,\quad a,b\in A, a\neq b.
    \end{equation}
\end{itemize}
 Then
  \begin{itemize}
  \item[$\bullet$] $\pi$~is irreducible as a $\Gamma$-representation.
  \item[$\bullet$] Up to equivalence, $\iota_+$ and $\iota_-$ are the only perfect
    boundary realizations of~$\pi$.
  \item[$\bullet$] Any imperfect boundary realization of~$\pi$ is equivalent to\\
    $\sqrt{t_+}\,\iota_+~\oplus~\sqrt{t_-}\,\iota_-:\mcH\to\mcH'_+\oplus\mcH'_-$
    for constants $t_+,t_->0$ with $t_++t_-=1$.
  \end{itemize}
\end{dthm}
\begin{othm}\label{OddTheorem}
Let $\pi:\Gamma\to\mcU(\mcH)$ be a unitary
representation of~$\Gamma$. Suppose
  \begin{itemize}
  \item[$\bullet$] $(\pi',\mcH')$ is an irreducible $\cp$-representation.
  \item[$\bullet$] $\iota:\mcH\to\mcH'$ is an imperfect realization of~$\pi$.
  \item[$\bullet$] The following Finite Trace Condition  holds
    \begin{equation*}\tag{FTC}
       \norm{(\iota^* \pi'(\one_{a})\iota)
      (\iota^* \pi'(\one_{b})\iota)}_{HS}<\infty,\quad a,b\in A, a\neq b.
    \end{equation*}
  \end{itemize}
  Then
  \begin{itemize}
  \item[$\bullet$] $\pi$~is irreducible as a $\Gamma$-representation.
  \item[$\bullet$] Up to equivalence, $\iota$ is the only boundary realization
    of~$\pi$.
  \end{itemize}
  Observe that the unitary $\Gamma$-action which $\pi'$ gives
  on~$\mcH'$ stabilizes $\mcH_1=\iota(\mcH)$, so it also stabilizes
  the orthogonal complement $\mcH_2=\mcH\ominus\mcH_1$. Let
  $\pi_2:\Gamma\to\mcU(\mcH_2)$ denote the $\Gamma$-action
  on~$\mcH_2$. One can also conclude:
  \begin{itemize}
  \item[$\bullet$] $\pi_2$ is irreducible.
  \item[$\bullet$] $\pi_2$ is inequivalent to $\pi$.
  \end{itemize}
\end{othm}

In order to use the above cited  Theorems  we need to investigate
the Finite Trace
Condition \eqref{FTC} for multiplicative representations.

We have the fundamental
\begin{lem}\label{fundamental}
 Let $(V_a,H_{ba},B_a)$ be a normalized
irreducible system and consider its twin $(\widehat{V}_a,\widehat{H}_{ba},\widehat{B}_a).$
 Assume the existence of
  a tuple of linear maps
$Q_a:V_a\to\widehat{V}_a$ satisfying \eqref{main}.

Let $J$ be the intertwining operator defined in \eqref{J}. Then, for every
$a\neq b\in~A,$ the operator $\w\pi(\one_b) J\pi(\one_a)$ has finite rank.
\end{lem}
\begin{proof}
We claim that, for fixed $a\neq b,$
$\widehat{\pi}({\bf{1}}_{b})\,J\, \pi({\bf{1}}_{a})$
maps $\mcH^{\infty}$ in a
subspace of $\widehat{\mcH}^{\infty}$ of  finite dimension.
Since finite dimensional subspaces are closed, this implies that
$\widehat{\pi}({\bf{1}}_{b})\,J\, \pi({\bf{1}}_{a})$ has finite rank.
More precisely, let
$\widehat{\mathcal{E}}_b$ be
the subspace of $\widehat{\mcH}^{\infty}$ generated by the set of (equivalence class of)
multiplicative functions
 $\{\widehat{\mu}[e,b,\widehat{v}_b],\;v_b\in \widehat{V}_b\}.$

It is clear that $\widehat{\mathcal{E}}_b$ is itself a finite dimensional space.
We are going to show that
\begin{equation}\label{contt}
\widehat{\pi}({\bf{1}}_{b})\,J\, \pi({\bf{1}}_{a})(\mcH^{\infty})\subset\widehat{\mathcal{E}}_b.
\end{equation}

By linearity and modulo the equivalence relation, it is sufficient to prove \eqref{contt}
for functions like $\mu[x,xc,v_c]$ with $|xc|=|x|+1.$

If $x\notin\Gamma(a),$ we have ${\bf{1}}_{\Gamma(a)}\mu[x,xc,v_c]=0,$ therefore
$$\widehat{\pi}({\bf{1}}_{b})\,J\, \pi({\bf{1}}_{a})(\mu[x,xc,v_c])=
\widehat{\pi}({\bf{1}}_{b})\,J\, ({\bf{1}}_{\Gamma(a)}\mu[x,xc,v_c])=0.$$

If $x\in\Gamma(a),$ we have ${\bf{1}}_{\Gamma(a)}\mu[x,xc,v_c]=\mu[x,xc,v_c],$ and,
since $a\neq b,$
\begin{eqnarray*}
& &\widehat{\pi}({\bf{1}}_{b})\,J\, \pi({\bf{1}}_{a})(\mu[x,xc,v_c])=
\widehat{\pi}({\bf{1}}_{b})\,J\, (\mu[x,xc,v_c])\\
&=&{\bf{1}}_{\Gamma(b)}(\widehat{\mu}[x,xc,-Q_c v_c]+\widehat{\mu}[xc,x,B_c v_c])
={\bf{1}}_{\Gamma(b)}(\widehat{\mu}[xc,x,B_c v_c]).
\end{eqnarray*}

It is easy to show by induction on $|x|$ that $x\in\Gamma(a)$ implies,
for all $w\in \widehat{V}_{c^{-1}},$
${\bf{1}}_{\Gamma(b)}(\widehat{\mu}[xc,x,w])\in \widehat{\mathcal{E}}_b.$

Indeed, if $|x|=1,$ then $x=a,$ and since $ac$ is reduced
\begin{eqnarray*}
& &{\bf{1}}_{\Gamma(b)}(\widehat{\mu}[ac,a,w])
=
{\bf{1}}_{\Gamma(b)}(\sum_{a'\neq c}\widehat{\mu}[a,a a',\widehat{H}_{a' c^{-1}}w])\\
&=&
{\bf{1}}_{\Gamma(b)}(\widehat{\mu}[a,e,\widehat{H}_{a^{-1} c^{-1}}w])
=
{\bf{1}}_{\Gamma(b)}(\sum_{a'\neq a}\widehat{\mu}[e,a',\widehat{H}_{a' a^{-1}}\widehat{H}_{a^{-1} c^{-1}}w])\\
&=&
{\bf{1}}_{\Gamma(b)}(\widehat{\mu}[e,b,\widehat{H}_{b a^{-1}}\widehat{H}_{a^{-1} c^{-1}}w])
=
\widehat{\mu}[e,b,\widehat{H}_{b a^{-1}}\widehat{H}_{a^{-1} c^{-1}}w]\in \w{\mathcal{E}}_b.
\end{eqnarray*}

Next we suppose the statement is true for $|x|=N-1$ and we consider $|x|=N,$ $x=x_1\dots x_N.$
Repeating twice the previous argument, we have
if $|xc|=|x|+1,$ by the induction hypothesis
\begin{eqnarray*}
& &{\bf{1}}_{\Gamma(b)}(\widehat{\mu}[xc,x,w])
=
{\bf{1}}_{\Gamma(b)}(\widehat{\mu}[x,x x_{N}^{-1},\widehat{H}_{x_{N}^{-1} c^{-1}}w])\\
&=&
{\bf{1}}_{\Gamma(b)}(\widehat{\mu}[x_1\dots x_{N-1},x_1\dots x_{N-2},
\widehat{H}_{x_{N-1}^{-1} x_N^{-1}}\widehat{H}_{x_{N}^{-1} c^{-1}}w])\in
\widehat{\mathcal{E}}_b.
\end{eqnarray*}
\end{proof}

We shall now consider the case of inequivalent twin systems:

\begin{thm}\label{irre}
  Let $(V_a,H_{ba},B_a)$ be a normalized
irreducible system and let $(\widehat{V}_a,\widehat{H}_{ba},\widehat{B}_a)$
be its twin.
Assume that the two systems are not equivalent. Assume moreover that the
$\G$-representations $\pi$ and $\w{\pi}$ arising from $(V_a,H_{ba},B_a)$ and
$(\widehat{V}_a,\widehat{H}_{ba},\widehat{B}_a)$ are equivalent.
Then $\pi$ satisfies duplicity. In particular
$\pi$ is irreducible as $\G$-representation.
\end{thm}

\begin{proof}
  Let  $\iota_{+}=\Id:\mcH\rightarrow\mcH,$
and $\iota_{-}=\Id\circ J:\mcH\rightarrow\widehat{\mcH}$ where
$J:\mcH\to\widehat{\mcH}$ is as in \eqref{J}.
  It is clear that $\iota_+$ is perfect and, since $J$ is a bijection, the same is true for $\iota_-$
(see Proposition \ref{perfect-itself}).
Since both $\iota_-$ and $\iota_+$ are unitary operators the Finite Trace Condition becomes simply:
$$
\| \widehat{\pi}({\bf{1}}_{b})\iota_-\iota_+^*
\pi({\bf{1}}_{a})\|_{HS}<+\infty\,,
$$
for $a,b\in A,$ $a\neq b.$
Since $\iota_-\iota_+^*$ is nothing but $J$ the result follows from
Lemma~\ref{fundamental} and the Duplicity Theorem~\ref{DupTheorem}.
\end{proof}

Let us turn now to equivalent twin systems:
\begin{thm}\label{irre-odd}
 Let $(V_a,H_{ba},B_a)$ be a normalized
irreducible system and let $(\widehat{V}_a,\widehat{H}_{ba},\widehat{B}_a)$
be its twin.
Assume that the two systems are equivalent. Then $\pi$ splits into the direct
sum of two irreducible $\G$-representations both satisfying oddity.
  \end{thm}
\begin{proof}
  Let, as in Remark~\ref{autovalori},
 $$
  \tilde{J}= \frac2{\sqrt{4-c^2}}(J-\frac{ic}2\mcK)
  $$
  where $J$ is the intertwining operator defined in~\ref{J}.
  It is more convenient to work with
  $$
  \mcJ=\inv\mcK \tilde{J}
  $$
   since its eigenvalues are $\pm1$.
  Let $\mcH_1$, $\mcH_2$, be the eigenspaces corresponding to $1$ and $-1$.
  One has $\mcH=\mcH_1\oplus\mcH_2$. Let $\pi_1$ be the restriction of $\pi$
  to $\mcH_1$.
  It is obvious that the inclusion map $\iota:
  \mcH_1\hookrightarrow\mcH$
  gives an imperfect realization for
  $\pi_1$.
  We shall prove that $\pi_1$ satisfies the hypothesis of the
  Oddity Theorem~\ref{OddTheorem}.
  For every $a\neq b$ we need to compute
 \begin{equation*}
       \norm{(\iota^* \pi(\one_{a})\iota
      \iota^* \pi(\one_{b})\iota)}_{HS}\;.
    \end{equation*}
 Observe that $\iota\iota^*$ is the projection onto $\mcH_1$, which is equal
 to $(\operatorname{Id}+\mcJ)/2$. Since $\pi(\one_a)\pi(\one_b)=0$ we have
 \begin{eqnarray}
   \pi(\one_{a})\iota
      \iota^* \pi(\one_{b})\iota)&=&\pi(\one_a)\left(\frac{\ID+\mcJ}2\right)\pi(\one_b)=
   \pi(\one_a)\left(\frac{\mcJ}2\right)\pi(\one_b) \nonumber\\
   &= & \frac12 \pi(\one_a)\inv\mcK\tilde{J}\pi(\one_b)\label{unoaJunob}\;.
  \end{eqnarray}
 The operator in the last line \eqref{unoaJunob}
 is a scalar multiple of
 $$
 \pi(\one_a)\left(\inv\mcK J- \frac{ic}2\ID\right)\pi(\one_b)=
 \pi(\one_a)\left(\inv\mcK J\right)\pi(\one_b).
 $$
 Since $\mcK$ intertwines
 $\pi$ to $\w\pi$ we have
 $$
 \pi(\one_a)\left(\inv\mcK J\right)\pi(\one_b)
 =\inv\mcK\w\pi(\one_a) J\pi(\one_b)\;.
 $$

 By Lemma~\ref{fundamental}, $\w\pi'(\one_a) J\pi'(\one_b)$ is a finite rank
 operator, and the same is true for each of the operators appearing
 at every step, up to the operator
 $\iota^* \pi'(\one_{a})\iota\iota^* \pi'(\one_{b})\iota,$
 so that its Hilbert--Schmidt norm is finite.
 Apply now Theorem~\ref{OddTheorem}.
\end{proof}
\begin{rem}
  The other representation arising from the eigenspace of $-1$ of $\mcJ$
  also satisfies the hypothesis of the Oddity Theorem: one can consult
  \cite{HKS19} or can calculate directly the FTC condition since
  the projection onto the other eigenspace is $(\ID-\mcJ)/2$.
\end{rem}

\begin{thm}\label{equi2}
 Assume that $\pi_1$ and $\pi_2$
  are multiplicative representations  built up from two irreducible
  inequivalent systems
  $(V^1_a,H^1_{b a},B^1_a)$
  and $(V^2_a,H^2_{b a},B^2_a)$.
  Assume that $\pi_1$ and $\pi_2$ are equivalent as $\G$-rep\-res\-enta\-tions.
  Then the two systems are twin.
\end{thm}
\begin{proof}
  Assume first that at least one of the two systems,
  let us say $(V^1_a,H^1_{b a},B^1_a),$  is not equivalent to its twin.

  As per Proposition \ref{perfect-itself} each $\pi_i$ ($i=1,2$)
  provides a perfect realization of itself. Assume that $T:\mcH_1\to\mcH_2$
  intertwines $\pi_1$ to $\pi_2$.

Define
$\iota:\mcH_1\rightarrow \mcH_1\oplus {\mcH}_2$  by
$\iota=\frac{1}{\sqrt{2}}(\operatorname{Id}\oplus\operatorname{Id}T)$
and $\pi'=\pi_1\oplus\pi_2$.

Let $\mcH'$ be the closure in $\mcH_1\oplus\mcH_2$
of the image of $\pi'(C(\Om))[\pi'(\mcH_1)]$.

It follows that $\iota$ is a boundary realization for $\pi_1$.
Moreover $(\pi',\mcH')$ will be perfect if and only if the operator $T$,
which intertwines
$\pi_1$ to ${\pi}_2$ intertwines also the two actions of $C(\Om)$
on $\mcH_1$ and $\mcH_2$.  Since this is impossible because $\pi_1$ and $\pi_2$
are inequivalent as $\cp$ representations,
Proposition \ref{impimp} applied to $\pi_1$ ensures that $\mcH_1$ contains
some nonzero vector satisfying \eqref{GVBo}.

On the other hand we know from
Theorem \ref{pregressi} that in this case there exists a tuple of linear
maps $Q_a:V^1_a\to \w{V^1}_a$ satisfying \eqref{main}.

By Theorem \ref{equi}
$\pi_1$ and $\w{\pi_1}$ are equivalent, by Lemma \ref{fundamental} they also
satisfy (FTC).

By Theorem \ref{DupTheorem}
$(\operatorname{Id},\mcH_1)$ and $(\operatorname{Id}J,\w{\mcH}_1)$
are the only perfect boundary realizations, hence the latter must be
equivalent to $(\operatorname{Id}T,{\mcH}_2)$.

Assume now that both systems $(V^1_a,H^1_{b a},B^1_a)$
and $(V^2_a,H^2_{b a},B^2_a)$
are equivalent to their twin: we shall see that this is impossible.

If $(V^1_a,H^1_{b a},B^1_a)$ is equivalent to its twin
by Theorem~\ref{irre-odd}, $\pi_1$ splits into the direct sum of two irreducible
inequivalent $\G$-representations $\pi_1^\pm$ and the same is true for $\pi_2.$
Since $\pi_1$ and $\pi_2$ are equivalent, one of the addends of $\pi_1$, say
$\pi_1^+$ is equivalent to one of the addends of $\pi_2$, say $\pi_2^+$.
By the Oddity Theorem~\ref{OddTheorem} $\pi_1^+$ and $\pi_2^+$ admit
\textbf{exactly one } boundary realization, and this implies that
$(V^1_a,H^1_{b a},B^1_a)$
and $(V^2_a,H^2_{b a},B^2_a)$ are equivalent: a contradiction.
\end{proof}

%


%
%



\end{document}